\documentclass[preprint,11pt]{amsart}%
\usepackage[margin=1.2in]{geometry}
\usepackage{amsmath}
\usepackage{amssymb}
\usepackage{amsfonts}
\usepackage{upref,amsthm,amsxtra,exscale}
\usepackage{cite}
\usepackage{amssymb}
\usepackage{amsmath}
\usepackage{graphicx,color}
\usepackage{t1enc}
\usepackage[utf8]{inputenc}
\usepackage{wasysym}
\usepackage{amsthm}

\usepackage{latexsym}
\usepackage{subfigure}
\usepackage{tikz}
\usepackage{float}
\usepackage{graphicx,color}

\usepackage[colorlinks=true,urlcolor=blue,
citecolor=red,linkcolor=blue,linktocpage,pdfpagelabels,
bookmarksnumbered,bookmarksopen]{hyperref}
\usepackage{epsfig,graphics,color}
\usepackage{graphicx}%
\newtheorem{theorem}{Theorem}[section]
\theoremstyle{plain}
\newtheorem{corollary}[theorem]{Corollary}
\newtheorem{lemma}[theorem]{Lemma}
\newtheorem{proposition}[theorem]{Proposition}
\newtheorem{remark}{Remark}
\numberwithin{equation}{section}

\begin{document}

\title[An Indefinite Elliptic Problem on $\mathbb{R}^N$]{\it An Indefinite Elliptic Problem on $\mathbb{R}^N$ Autonomous at Infinity:\\ The Crossing Effect of the Spectrum and the Nonlinearity}
\author{Liliane A. Maia}
\address{Departamento de Matem\'{a}tica, UNB, 70910-900 Bras\'{\i}lia, Brazil.}
\email{lilimaia@unb.br}
\author{Mayra Soares}
\address{Departamento de Matem\'{a}tica, PUC-Rio, 22451-900 Rio de Janeiro, Brazil.}
\email{ssc\_mayra@hotmail.com}
\thanks{FAPDF 0193.001300/2016, 0193.001765/2017, CNPq 308378/2017-2.}
\date{\today} 
\maketitle 
\vspace{-0.3cm}
\begin{changemargin}{-1.1cm}{-1.1cm}  
\quad \quad \quad \ \rule{15.3cm}{0.001cm} 
\begin{abstract}
{\small We present a new approach to solve an indefinite Schr\"odinger Equation autonomous at infinity, by identifying the relation between the arrangement of the spectrum of the concerned operator and the behavior of the nonlinearity at zero and at infinity.  The main novelty is how to set a skillful linking structure that overcome the lack of compactness, depending on the growth of the nonlinear term and making use of information about the autonomous problem at infinity. Here no monotonicity assumption is required on the nonlinearity, which may be sign-changing as well as the potential. Furthermore, depending on the nonlinearity, the limit of the potential at infinity may be non-positive, so that zero may be an interior point in the essential spectrum of the Schr\"odinger operator.}
\bigskip
\newline
\textsc{Key words: }{\small Nonlinear Schr\"{o}dinger Equation; Autonomous at Infinity; Asymptotically Linear; Variational Methods; Spectral Theory; Linking Structure.}
\bigskip
\newline
\textsc{MSC 2010: } {\small 35j10, 35j20, 35j60, 35p05.}\medskip
\newline	
\rule{15.3cm}{0.001cm}
\end{abstract}
\end{changemargin}
\maketitle 
\vspace{0.5cm}
\section{Introduction}
\label{sec:introduction}

\qquad In this paper we establish the existence of a nontrivial weak solution to the following problem
\begin{equation}
\left\{
\begin{array}
[c]{l}%
-\Delta u+V(x)u=f(u),\\
u\in{H}^{1}(\mathbb{R}^{N}),
\end{array}
\right.  \tag{$P_V$}\label{prob}%
\end{equation}
for $N\geq3$, where $V$ and $f$ may change sign and $f$ is an asymptotically linear nonlinearity. Problem (\ref{prob}) has been studied
extensively in order to reach the most general hypotheses which grant to solve it. Our goal is to expose the essential interplay between the spectrum of the Schr\"odinger operator and the nonlinear term, which enable us to waive  assumptions of decay and differentiability of $V$ and $f$, or even the monotonicity of $f(s)/s$, as those required in \cite{FMM,SS,JT2,CT,MOR}, for instance.

\qquad The new idea is to use the complete knowledge of the spectrum determined by $V$ and discover the nonlinearities which interact appropriately with this spectrum in order to generate a linking geometry that leads to find a solution. Our approach is different from that applied to this problem on a bounded domain $\Omega$, where the spectrum is discrete and the  classical linking geometry is clearly built by choosing an eigenfunction, see \cite{R,CM}. In fact, exploiting spectral properties we present a linking structure by choosing a special direction $e$, which is not necessarily an eigenfunction. Furthermore, due to the lack of monotonicity on the nonlinearity, a new construction of linking level depending on a continuous function $h_0$ is developed (see (\ref{c}) and (\ref{h_0})) in order to succeed in finding a nontrivial critical point to the functional associated to the problem, despite the lack of compactness. This argument is innovative and intricate since it combines the dilation, translation and projections of a positive solution of the limit autonomous problem.  This approach could be adapted to other situations in order to improve results in literature where the monotonicity have been required.

\qquad We assume the following hypotheses of behavior on the nonlinear term:
\medskip\newline
$(f_1)\quad
f\in\mathcal{C}(\mathbb{R})$ is odd, $f(0)=0$, there exists $\displaystyle\lim_{s \to 0}\dfrac{f(s)}{s}=f'(0)$ and 
\[
\dfrac{f(s)}{s} \geq f'(0) \quad \text{for \ all\ } s\in \mathbb{R}^+;
\]
$(f_2)$\quad There exists $a>f'(0)$, such that $\displaystyle\lim_{|s| \to +\infty}\dfrac{f(s)}{s}=a$.
\medskip
\newline

\qquad Defining  
operator $A:= -\Delta + V(x)$,
as an operator of $L^2(\mathbb{R}^N)$ and denoting by $\sigma(A)$ the spectrum, by $\sigma_p(A)$ the point spectrum  and by $\sigma_{ess}(A)$ the essential spectrum of $A$, respectively, we assume the following spectral hypothesis on $V$: 
\medskip
\newline
$(V_1)$\quad $(f'(0), a)\cap \sigma(A)\not= \emptyset$ and $ f'(0), a\notin \sigma_p(A)$;
\medskip
\newline
$(V_2)$\quad $V\in C(\mathbb{R}^N,\mathbb{R})$ and
$\displaystyle\lim_{|x| \to +\infty}V(x) =  V_\infty > f'(0)$.
\medskip


\qquad Assumption $(V_2)$ implies that $A$ is a self-adjoint operator. Moreover, the limit in $(V_2)$ implies that ${\sigma_{ess}(A) = [V_\infty,+\infty)}$, and since  ${f'(0) < V_\infty}$ we conclude that $ f'(0) \notin \sigma_{ess}(A)$, by simplicity, we also suppose $f'(0)\notin\sigma_{p}(A),$ hence $f'(0)\notin \sigma(A)$. In addition, $\sigma(A)\cap(-\infty,V_\infty)$ is the discrete spectrum, composed at most by a sequence of eigenvalues with finite multiplicity.

 \qquad Assumption $(V_1)$ is a kind of crossing hypothesis, since it requires some intersection between the spectrum of $A$ and the interval $(f'(0),a)$. Note that if ${\sigma^+\in (f'(0),a)\cap \sigma(A)}$ is such that $(f'(0),\sigma^+)\cap \sigma(A) = \emptyset$, then $\sigma^+$ is either one of the eigenvalues in the discrete spectrum, or $\sigma^+ = V_\infty$. Moreover, if $\sigma^-<f'(0)$ is such that $(\sigma^-,f'(0))\cap \sigma(A)= \emptyset$ and $\sigma^-$ is either one of those eigenvalues, or $\sigma^- = -\infty$, then $(\sigma^-,\sigma^+)$ is a gap in $\sigma(A)$, which contains $f'(0)$. Furthermore, if $a>V_\infty$, then $a \in \sigma_{ess}(A)$, namely, it is not necessary to require $a \notin \sigma_p(A)$. However, if $\sigma^+ <a<V_\infty$, then we require that $a$ cannot be one of the eigenvalues of $A$.
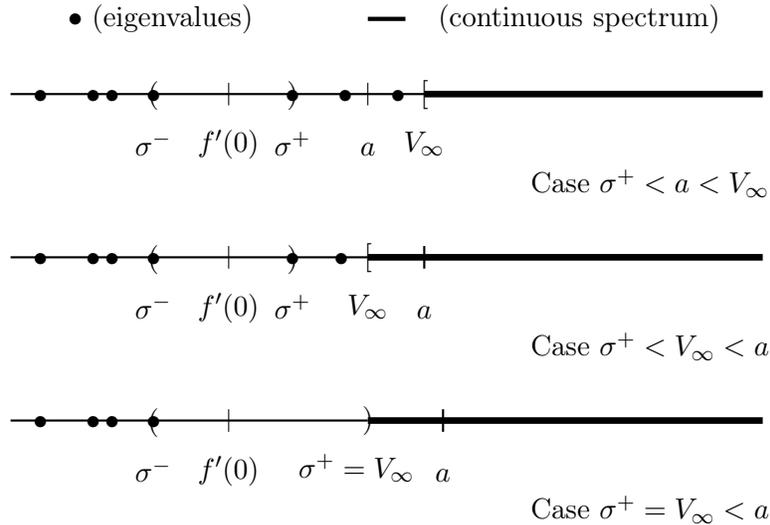
\begin{figure}[h]
	\centering
	\begin{tikzpicture}[scale=.50]
	\draw[thick, -] (-2, 3) -- (9, 3);
	\draw (-1.2,2.95) node[]{$\bullet$};
	\draw (0.7,2.95) node[]{$\bullet$};
	\draw (0.2,2.95) node[]{$\bullet$};
	\draw (1.8,2.95)node[]{$($};
	\draw (1.8,1.1) node[above]{$\sigma^-$};
	\draw (3.8,1.0) node[above]{$f'(0)$};
	\draw[, -] (3.8, 3.3) -- (3.8,2.7);
	\draw (1.8,2.95) node[]{$\bullet$};
	\draw (5.5,1.1) node[above]{$\sigma^+$};
	\draw (5.5,2.95) node[]{$\bullet$};
	\draw (5.5,2.95)node[]{$)$};
	\draw (6.9,2.95) node[]{$\bullet$};
	\draw (7.5,1.1) node[above]{$a$};
	\draw (8.3,2.95) node[]{$\bullet$};
	\draw[ultra thick,line width=2.5pt] (9,3)--(18,3);
	\draw (9,3) node[]{$[$};
	\draw (9,1.1) node[above]{$V_\infty$};
	\draw [,-] (7.5,3.3) -- (7.5,2.7);
	\draw (15,0) node[above]{\textrm{Case $\sigma^+<a<V_\infty $}};
	\draw (2,5) node[] {\textrm{$\bullet$  (eigenvalues)}};
	\draw (13,5) node[] {\textrm{ (continuous spectrum)}};
	\draw[ultra thick,-] (7.5,5)--(8.5,5);
	\end{tikzpicture}
\end{figure}
\vspace{-0.5cm}
\begin{figure}[h]
	\centering
	\begin{tikzpicture}[scale=.50]
	\draw[thick, -] (-2, 3) -- (7.5, 3);
	\draw (-1.2,2.95) node[]{$\bullet$};
	\draw (0.7,2.95) node[]{$\bullet$};
	\draw (0.2,2.95) node[]{$\bullet$};
	\draw (1.8,2.95)node[]{$($};
	\draw (1.8,1.1) node[above]{$\sigma^-$};
	\draw (3.8,1.0) node[above]{$f'(0)$};
	\draw[, -] (3.8, 3.3) -- (3.8,2.7);
	\draw (1.8,2.95) node[]{$\bullet$};
	\draw (5.5,1.1) node[above]{$\sigma^+$};
	\draw (5.5,2.95) node[]{$\bullet$};
	\draw (5.5,2.95)node[]{$)$};
	\draw (6.8,2.95) node[]{$\bullet$};
	\draw (7.5,1.1) node[above]{$V_\infty$};
	\draw[ultra thick,line width=2.5pt] (7.5,3)--(18,3);
	\draw (7.5,3) node[]{$[$};
	\draw (9,1.1) node[above]{$a$};
	\draw [thick,-] (9,3.3) -- (9,2.7);
	\draw (15,0) node[above]{\textrm{Case $\sigma^+<V_\infty  < a$}};
	\end{tikzpicture}
	\end{figure}
\vspace{-0.5cm}
\begin{figure}[h]
		\begin{tikzpicture}[scale=.50]
	\draw[thick, -] (-2, 3) -- (7.5, 3);
	\draw (-1.2,2.95) node[]{$\bullet$};
	\draw (0.7,2.95) node[]{$\bullet$};
	\draw (0.2,2.95) node[]{$\bullet$};
	\draw (1.8,2.95)node[]{$($};
	\draw (1.8,1.1) node[above]{$\sigma^-$};
	\draw (3.8,1.0) node[above]{$f'(0)$};
	\draw[, -] (3.8, 3.3) -- (3.8,2.7);
	\draw (1.8,2.95) node[]{$\bullet$};
	\draw (7.2,1.1) node[above]{$\sigma^+=V_\infty$};
	\draw[,line width=2.5pt] (7.5,3)--(18,3);
	\draw (7.5,3) node[]{$)$};
	\draw (9.5,1.1) node[above]{$a$};
	\draw [thick,-] (9.5,3.3) -- (9.5,2.7);
	\draw (15,0) node[above]{\textrm{Case $\sigma^+=V_\infty  < a$}};
	\end{tikzpicture}
	\end{figure}


 \qquad It is important to highlight that the relation between $\sigma(A)$ and $f$ is established by the limits of $\dfrac{f(s)}{s}$ at the origin and at infinity, which must be in a spectral gap and after the same spectral gap, respectively. For instance, it is straightforward to verify that the asymptotically linear
model nonlinearity $f(s)=\frac{as^{3} -s }{1+s^{2}}$ satisfies hypotheses $(f_1)-(f_2)$, with $f'(0)=-1 <a$. In this case, as mentioned above, assumptions $(V_1)-(V_2)$ imply that $-1\notin \sigma(A)$, but it is not required that $0\notin \sigma(A)$ as usual, see \cite{CT}. Moreover, hypothesis $(V_1)$ implies that $a > \sigma^+> -1$, but does not require $a>0$ as usual, see \cite{JT}. Finally, we point out that potentials $V$ such that $V(x) \to V_\infty \leq 0$ as $|x|\to +\infty$ are included by hypothesis $(V_2)$, provided that $f'(0)<0$.

\qquad To the best of our knowledge, these assumptions on problem (\ref{prob}) deeply generalize previous works in the literature. Setting
 \[F(s) = \displaystyle\int_{0}^{s}f(t)\;dt\quad \text{and} \quad \hat{F}(s) := f(s)s - 2F(s), \quad \text{for \ all} \quad s \in \mathbb{R}.
\]
our main result is stated as follows.
\medskip
\begin{theorem}\label{main2}
	Assume that $(V_1)-(V_2)$, $(f_1)-(f_2)$ hold with either \newline
	$(f_3)$\quad There exists $\delta>0$ such that $\dfrac{f(s)}{s} \leq V_\infty - \delta$ for all $s \in \mathbb{R}$, 
	\medskip
	\newline
	or 
	\medskip
	\newline
	$(f_3)'$\quad $\hat{F}(s)\geq 0$ and there exists $\delta>0$ such that
	\[
	\dfrac{f(s)}{s}> V_\infty - \delta \implies \hat{F}(s)\geq \delta,
	\]	
	along with
	\medskip
	\newline
	$(V_3)$\quad $V(x)\leq V_\infty.$
	\medskip
	\newline
	Then problem (\ref{prob}) has a nontrivial solution.	
\end{theorem}	
\medskip

	 \qquad Hypotheses $(f_3)$ or $(f_3)'$ were first introduced by L. Jeanjean and K. Tanaka in \cite{JT} and are required for obtaining the boundedness of Cerami sequences for the functional associated to $(P_V)$ in the variational approach. Note that hypothesis $(f_3)$ does not make sense when $V_\infty =\sigma^+ <a$. Indeed, since $f(s)/s\to a$ as $|s| \to +\infty$, then there is no $\delta>0$ such that $f(s)/s<a - \delta$ for $|s|$ sufficiently large. On the other hand, $(f_3)$ may happen when $V_\infty > a >\sigma^+$. Furthermore, we observe that if there exists a sequence of eigenvalues in the discrete spectrum converging to $V_\infty$, the bottom of the essential spectrum, then we surely have $\sigma^+< V_\infty$, else, both cases $\sigma^+<V_\infty$ and $\sigma^+ = V_\infty$ may happen.

\qquad Influenced by the existence results obtained in \cite{JT,JT2} by L. Jeanjean and K. Tanaka for definite Schr\"odinger equations by applying Mountain Pass Theorem, we analyze how to treat the existence of solution to problem $(P_V)$ under the most general assumptions, as for indefinite potentials, for instance. In \cite{JT} the authors worked with a asymptotically linear problem satisfying
\[
f'(0)=0< \inf\sigma(-\Delta+V(x))<a,
\]
whereas in \cite{JT2} they treated the superlinear case requiring
\[
f'(0) < \inf\sigma(-\Delta + V(x)) < a = +\infty.
\]
Since we are interested in the asymptotically linear case, we generalize theirs assumptions requiring 
\begin{equation}\label{cross}
f'(0) < \sigma^+ < a < +\infty, \quad \text{for \ some} \quad \sigma^+ \in \sigma(-\Delta +V(x)),
\end{equation}
which includes both previous assumptions. Hence, in what concerns to release the interaction between the spectrum and the nonlinear term, our Theorem \ref{main2} is a wide generalization to Theorems 1.1 and 1.2 in \cite{JT}, since we only require the crossing assumption in (\ref{cross}). Regarding \cite{JT2}, although $\inf\sigma(-\Delta +V(x))$ may be negative, since $f'(0)$ is below this infimum, by making a translation in the problem, it becomes a definite problem, hence it is also possible to obtain a Mountain Pass geometry, see also \cite{LSW}. Thus, these papers work with definite operators, which allow the authors to find a positive solution to the problem. Since in our case, we treat especially indefinite operators, a solution $u$ to $(P_V)$ is an eigenvector associated to the eigenvalue zero of the following Schr\"odinger operator
\[
-\Delta + V(x) - \dfrac{f(u)}{u},
\]  
whose infimum of the spectrum is negative. Therefore, such a solution must be sign-changing as $\inf(-\Delta+V)<0$, for details see \cite{P}.

 \qquad Under our assumptions, for the indefinite case, Mountain Pass Theorem is not applicable, hence our challenge is to set up a suitable linking geometry to problem ($P_V$), which allows us to find a nontrivial critical point to the functional associated by the variational methodology. Following the approach in \cite{JT}, we look for information in the limit problem ``at infinity'', however, due to the lack of compactness our arguments do not work for a usual linking structure. Establishing an appropriate linking geometry depends strongly on the hypotheses of growth of $f(s)/s$, since we do not have assumptions of monotonicity, as well as \cite{JT}. Thus, we make use of classical results from \cite{BL} by H. Berestycki and P. Lions and we also lay hand of a smart construction by L. Jeanjean and K. Tanaka in Proposition 4.1 of \cite{JT}. Thereby, we get a proper linking structure to overcome the lack of compactness.
 Nevertheless, another difficulty is to guarantee that, in fact, this linking geometry produces a Cerami sequence for the associated functional. For this purpose, we need to prove a linking theorem, which do generalize the linking theorems in \cite{R,LW}, since our linking structure is not covered by these previous results.
 
 \qquad In \cite{MOR} problem $(P_V)$ was treated by L. Maia, J. Oliveira Junior and R. Ruviaro, also in the asymptotically linear case, autonomous at infinity. However, their estimates required more assumptions of decay, differentiability, growth and monotonicity, which enabled them to apply the linking theorem developed by G. Li and C. Wang in \cite{LW} and use the monotonicity to get the nontrivial critical point at the end. This argument cannot be employed by us, due to the lack of monotonicity. Furthermore, in \cite{MOR} the authors require $0\notin \sigma(-\Delta+V(x))$ and the limit of $V$ being positive, as usual in the literature, such hypotheses are loosened by us in this paper.

\qquad Finally, we would like to mention the paper \cite{LSW} by Z. Liu, J. Su and T. Weth, where they worked under general assumptions on the potential and nonlinear term, including the asymptotically linear case. Nevertheless, they required a non-crossing assumption, which in our case would mean 
\begin{equation}\label{noncross}
[f'(0),a]\cap \sigma_{ess}(A) \subset \left[\inf_{s\in \mathbb{R}\setminus\{0\}}\dfrac{f(s)}{s},\sup_{s\in \mathbb{R}\setminus\{0\}}\dfrac{f(s)}{s}\right]\cap\sigma_{ess}(A) = \emptyset \quad \text{and} \quad  a\notin \sigma(A).
\end{equation}
 Hence, their nonlinear term did not cross the essential spectrum, then they were able to get compactness and by making a translation in their problem they could apply Mountain Pass Theorem. We observe that hypothesis $(V_1)$ requires less then condition (\ref{noncross}), since we just assume $a\notin \sigma_p(A)$ and it does not matter how $f(s)/s$ behaves between zero and infinity, namely it may cross the essential spectrum or not. This shows, for instance, that a nonlinearity satisfying $(f_3)'$ cannot be tackled by their results, neither the cases $\sigma^+\leq V_\infty <a$ and $\sigma^+< V_\infty = a$. 

\qquad Following ideas found in \cite{JT2,LSW,MS4}, by making a translation in our problem, we observe that $u \in H^1(\mathbb{R}^N)$ is a solution to (\ref{prob}) if and only if $u$ is a solution to
\begin{equation}
\left\{
\begin{array}
[c]{l}%
-\Delta u+V_0(x)u=f_0(u),\\
u\in{H}^{1}(\mathbb{R}^{N}),
\end{array}
\right.  \tag{$P_{V_0}$}\label{prob_0}%
\end{equation}
where 
\begin{eqnarray}\label{def}
\quad \quad \quad \quad V_0(x) &:=& V(x)-f'(0)\nonumber\\
\quad \quad \quad \quad f_0(s) &:=& f(s) - f'(0)s, \; \text{for \ all} \; s\in \mathbb{R}.
\end{eqnarray}
\medskip
\newline
Hence, we are going to prove Theorem \ref{main2} by dealing with $(\ref{prob_0})$, which is more easygoing to treat, since it satisfies next hypotheses $(V_1)_0-(V_2)_0$, ${(f_1)_0-(f_2)_0}$ as a consequence of $(V_1)-(V_2),$ ${(f_1)-(f_2)}$ and either $(f_3)_0$ as a consequence of $(f_3)$, or $(f_3)'_0$ along with $(V_3)_0$ as a consequence of $(f_3)'$ along with $(V_3)$. Setting
 \[
F_0(s):=\displaystyle\int_{0}^{s}f_0(t)\,dt \quad  \text{and} \quad  \hat{F}_0(s) := f_0(s)s - 2F_0(s) = \hat{F}(s) \quad \text{for \ all} \quad s \in \mathbb{R},
\] we state these assumptions as follows.
\medskip\newline
$(f_1)_0\quad
f_0\in\mathcal{C}(\mathbb{R})$ is odd, $f_0(0)=0$, there exists $\displaystyle\lim_{s \to 0}\dfrac{f_0(s)}{s}= 0$ and ${f_0(s)} \geq 0$ for all $s\in \mathbb{R}^+;$
\medskip\newline
$(f_2)_0$\quad Setting $a_0:= a - f'(0)$, it holds 
\[
\displaystyle\lim_{|s| \to +\infty}\dfrac{f_0(s)}{s}=a_0;
\] 
\medskip
$(f_3)_0$\quad There exists $\delta>0$ such that  $\dfrac{f_0(s)}{s}\leq V_\infty - f'(0) - \delta =: V_{0,\infty} - \delta$ for all $s \in \mathbb{R}$;
\medskip
\newline
$(f_3)'_0$\quad $\hat{F}_0(s)\geq0$ and there exists $\delta>0$ such that
\[
\dfrac{f_0(s)}{s}> V_{0,\infty} -\delta \implies \hat{F}_0(s)\geq \delta.
\]

\qquad Defining $A_0 := A - f'(0) =  -\Delta + V_0(x)$ and denoting the spectrum of $A_0$ by $\sigma(A_0)$,
\bigskip
\newline
$(V_1)_0$\quad $\big(0, a_0\big)\cap \sigma(A_0)\not= \emptyset$ and ${0,  a_0\notin \sigma_p(A_0)}$ with ${(\sigma^- - f'(0),\sigma^+ - f'(0))}$ a spectral gap of $\sigma(A_0)$ such that $\sigma^+- f'(0) \in (0,a_0)\cap\sigma(A_0)$. 
\medskip
\newline
$(V_2)_0$\quad $V_0 \in C(\mathbb{R}^N,\mathbb{R})$ and $\displaystyle\lim_{|x| \to +\infty}V_0(x) = V_{0,\infty}>0$;
\newline
$(V_3)_0$\quad $V_0(x)\leq V_{0,\infty}.$
\medskip
\newline

\qquad Theorem \ref{main2} is going to be proved by applying variational methods. Let us briefly
highlight some technical details. Defining the Hilbert space
$E:=\Big(H^1(\mathbb{R}^N), ||\cdot||\Big)$, where $||\cdot||$ is the norm induced by operator $A_0$ and considering $\{\mathcal{E}(\lambda): \lambda \in \mathbb{R}\}$ as the spectral family of operator $A_0$, we set $E^+ \subset E$ as the subspace given by $E^+ := \Big(I - \mathcal{E}(0)\Big)E$, where $A_0$ is positive definite and its complement ${ E^- := \mathcal{E}(0)E}$, with $E^-$ the subspace where $A_0$ is negative definite, respectively, hence $E = E^+  \oplus E^-$. In view of $(V_1)_0$ there exists ${\big(\sigma^- - f'(0), \sigma^+-f'(0)\big)}$, which is a spectral gap of $A_0$, with ${\sigma^- - f'(0)<0< \sigma^+-f'(0)}$, then by the spectral family definition, one has
\begin{equation}\label{e00}
\int_{\mathbb{R}^N}\Big(|\nabla u^+(x)|^2+ V_0(x)(u^+(x))^2\Big)\;dx \geq \Big(\sigma^+- f'(0)\Big)\int_{\mathbb{R}^N}(u^+(x))^2\;dx,
\end{equation}
for all $u^+ \in E^+$ and
\begin{equation}\label{e000}
-\int_{\mathbb{R}^N}\Big(|\nabla u^-(x)|^2+ V_0(x)(u^-(x))^2\Big)\;dx \geq \Big(f'(0) - \sigma^-\Big) \int_{\mathbb{R}^N}(u^-(x))^2\;dx,
\end{equation}
for all $u^- \in E^-.$ 
Furthermore, since $f'(0)\notin \sigma(A)$, then $0 \notin \sigma(A_0)$ and $\ker(A_0)= \{0\}$ is the trivial subspace of $E$, 
hence the following inner product is well defined and we are focused on looking for a solution to (\ref{prob_0}) on the Hilbert space $E$ endowed with this suitable inner product

\begin{equation}\label{IP}
\big( u,v \big) = \left\{
\begin{array}{lllll}
\quad \displaystyle\int_{\mathbb{R}^N}\Big(\nabla u(x) \nabla v(x) + V_0(x)u(x)v(x)\Big)dx = (A_0u,v)_{L^2(\mathbb{R}^N)}  \quad   \quad \quad \quad \text{if} \ \ u, v \in E^+, \\
\\
\quad -\displaystyle\int_{\mathbb{R}^N}\Big(\nabla u(x) \nabla v(x) + V_0(x)u(x)v(x)\Big)dx = -(A_0u,v)_{L^2(\mathbb{R}^N)} \quad    \ \quad\text{if} \ \ u,v \in E^-,\\
\\
\quad 0 \quad \quad   \quad \ \quad \quad \quad \quad \quad \quad \quad \quad \quad \quad \quad  \quad  \quad \quad \quad \quad \quad \quad \quad \quad  \text{if} \ \ u \in E^i, \ v \in E^j, i\not=j,
\end{array}
\right.
\end{equation}
and the corresponding norm $||u||^2 := (u,u)$ for all $u \in E$, which is equivalent to the standard norm in $H^1(\mathbb{R}^N)$, see \cite{CT}, for instance.

\begin{remark}\label{r4}
Observe that hypothesis $(V_1)_0$ is sufficient to guarantee equivalence of the refereed norms, in view of (\ref{e00})-(\ref{e000}) and that $0 \notin \sigma(A_0)$. Actually, it is only necessary $0 \notin \sigma_{ess}(A_0)$, namely, we could have $0 \in \sigma_p(A_0)$. However, we could not weaken this assumption allowing $f'(0) \in \sigma_{ess}(A)$, namely, $0 \in \sigma_{ess}(A_0)$. In fact, if $0 \in \sigma_{ess}(A_0)$, then $\ker(A_0)$ is infinite dimensional and the inner product in (\ref{IP}) is necessarily not equivalent to the standard inner product in $H^1(\mathbb{R}^N)$.
\end{remark}

\qquad For the variational approach, we associate to (\ref{prob_0}) the functional $I:E\to \mathbb{R}$ given by
\begin{equation}\label{I}
I(u) = \dfrac{1}{2}\Big(||u^+||^2 -||u^-||^2 \Big) - \int_{\mathbb{R}^N}F_0(u(x))\,dx,
\end{equation}
which is indefinite and belongs to $C^1(E,\mathbb{R})$ in view of the previous hypotheses. We recall that $u$ is a weak solution to (\ref{prob_0}) if it is a critical point of $I$, namely, if
\[
I'(u)v = (u^+,v^+) - (u^-,v^-) - \int_{\mathbb{R}^N}f_0(u(x))v(x)\;dx = 0, \quad \forall \; v \in E. 
\]
Moreover, $(u_n)\subset E$ is called a Cerami sequence for $I$ if
\[
\displaystyle\sup_n|I(u_n)| <+\infty \quad \text{and} \quad ||I'(u_n)||_{E'}(1 + ||u_n||)\to 0 \quad \text{as} \quad n\to+\infty,
\]
and $(u_n)$ is called a $(C)_c$ sequence, or a Cerami sequence on the level $c$, if besides that it satisfies $I(u_n)\to c$ as $n\to +\infty$.

\qquad Our goal is to prove and apply an abstract linking theorem, which is going to be appropriate for our problem, adapting the linking results presented in \cite{LW} due to G. Li and C. Wang, and in \cite{R}, due to P. Rabinowitz, in order to provide a Cerami sequence for functional $I$ associated to $(P_{V_0})$ on a positive level $c \in \mathbb{R}$. Indeed, our abstract result is stated as follows. 

 \begin{theorem}[Abstract Linking Theorem]\label{AR}
	Let $E$ be a real Banach space with $E=V\oplus X$, where $V$ is finite dimensional. Suppose there exist real constants $R> \rho>0$, $\alpha>\omega$ and there exists an $e \in \partial B_1\cap X$ such that $I\in C^1(E,\mathbb{R})$ satisfies,
	\bigskip
	\newline
	$(I_1)$\quad $I|_{\partial B_\rho\cap X}\geq \alpha$;
	\medskip
	\newline
	$(I_2)$\quad Setting $Q:= (\bar{B}_R \cap V)\oplus \{re :0\leq r\leq R\}$, there exists an $h_0\in C(Q,E)$ such that
	\medskip 
	\newline
	$(i)\quad\displaystyle\sup_{u\in Q}I(h_0(u))<+\infty$,
	\newline
	$(ii)\quad\displaystyle\sup_{u\in \partial Q}I(h_0(u)) = \omega,$ 
	\medskip
	\newline
	$(iii)\quad h_0(\partial Q)\cap (\partial B_\rho \cap X) = \emptyset,$
	\medskip
	\newline
	$(iv)\quad$There exists a unique $w \in h_0(Q)\cap(\partial B_\rho \cap X)$ and $\deg(h_0,\text{int}(Q),w)\not=0$. 	
\medskip
\newline
Then $I$ possess a Cerami sequence on a level $c\geq \alpha$, which can be characterized as
\begin{equation}
\label{c}
c:= \inf_{h\in\Gamma}\max_{u \in Q} I(h(u)),
\end{equation}
where $\Gamma:= \{h \in C(Q,E):h|_{\partial Q} = h_0\}$.
\end{theorem}

\begin{remark}
We observe that, if $h_0 = Id,$ the identity map, then Theorem \ref{AR} is reduced to Theorem 2.10 in \cite{LW}. In this sense, our abstract result is a generalization for the abstract results developed in \cite{LW} and \cite{R}. In order to apply this result under our assumptions in (\ref{prob_0}), we are going to choose a suitable $h_0$, which sometimes cannot be the identity map.
\end{remark}
\medskip

\qquad Next section is devoted to prove Theorem \ref{AR}. Subsequently, in section $3$ we are going to apply Theorem \ref{AR} and get a $(C)_c$ sequence for $I$, named $(u_n)$. In virtue of the assumptions on $f$, in section $4$ we are going to be able to show that $(u_n)$ is bounded. Finally, section $5$ is designated to obtain a nontrivial critical point for $I$, which is going to be a nontrivial solution to (\ref{prob_0}). More specifically, by means of an indirect argument, lying down on information about the autonomous problem at infinity, we are going to provide a nontrivial critical point for $I$, which is going to be the weak limit of the $(C)_c$ sequence obtained for $I$ by means of Theorem \ref{AR}.

\section{An Appropriate Abstract Linking Result}

\qquad In this section we aim  to prove Theorem \ref{AR} and to do so, we first recall two auxiliary results, which can be found in \cite{LW}. For sake of simplicity we change their notation to suit this result in our current notation.

\begin{proposition}[Proposition 2.8 \cite{LW}]\label{P2.8}
Let $E$ be a Banach space and $Q$ a metric space. Let $Q_0$ be a closed subspace of $Q$ and $\Gamma_0 \subset C(Q_0,E)$. Define 
\[
\Gamma := \{h\in C(Q,E):h|_{Q_0}\in\Gamma_0\}.
\]
If $I\in C^1(E,\mathbb{R})$ satisfies
\begin{equation}\label{s2e1}
+\infty > c := \inf_{h\in\Gamma}\sup_{u \in Q}I(h(u)) > \omega := \sup_{h_0\in\Gamma_0}\sup_{u\in Q_0}I(h_0(u)),
\end{equation}
then for every $\varepsilon\in \left(0,\dfrac{c-\omega}{2}\right)$, $\delta>0$ and $h\in \Gamma$ such that
\[
\sup_{u\in Q}I(h(u))\leq c+ \varepsilon,
\]
there exists $u\in E$ such that
\medskip
\newline
$a) \quad c-2\varepsilon \leq I(u)\leq c+2\varepsilon$,
\medskip
\newline
$b) \quad \text{dist}(u,h(Q)) \leq 2\delta,$
\newline
$c) \quad (1+||u||_E)||I'(u)||_{E'}<\dfrac{8\varepsilon}{\delta}$.
\end{proposition}	
\medskip	 
\begin{corollary}[Corollary 2.9 \cite{LW}] \label{C2.9}
	 Under the assumptions of Proposition \ref{P2.8}, there exists a sequence ${(u_n)\subset E}$ satisfying
	 \[
	 I(u_n) \to c, \quad (1+||u_n||_E)||I'(u_n)||_{E^*} \to 0.
	 \]
\end{corollary}	

\qquad With the purpose of proving Theorem \ref{AR}, we are going to show that under its assumptions we can apply Proposition \ref{P2.8} and Corollary \ref{C2.9}, obtaining the desired $(C)_c$ sequence for $I$.
\medskip

\begin{proof}[Proof of Theorem \ref{AR}]
In order to apply Proposition \ref{P2.8}, we first show that ${c\geq \alpha>\omega}$. In fact, we prove that for every $h \in \Gamma$ there exists $v_h \in h(Q)\cap (\partial B_\rho\cap X)$. Hence, in view of $(I_1)$ we arrive at
	\[
	\max_{u \in Q}I(h(u))\geq I(v_h)\geq \alpha, \quad \forall \quad h \in \Gamma,
	\]
	getting $c\geq\alpha$. Let us obtain such a $v_h$. Observe that we look for $u_h \in Q$ such that ${h(u_h) = v_h \in X}$ and $||v_h|| = \rho$. Defining $P:E\to V$ as the projector operator on $V$ and $H:Q\to \mathbb{R}e \oplus V$ given by
	\[
	H(u) := ||(1-P)h(u)||e + P(h(u)), 
	\]
	it implies that $H\in C(Q,\mathbb{R}e \oplus V)$ and provided that $h|_{\partial Q} = h_0$, for $u\in \partial Q$ we have
	\[
	H(u) = ||(1-P)h_0(u)||e + P(h_0(u)) =: H_0(u)\not= \rho e,
	\]
	since $h_0(\partial Q)\cap (\partial B_\rho \cap X) = \emptyset$, from $(iii)$ in $(I_2)$. Thus, identifying $\mathbb{R}e \oplus V$ with $\mathbb{R}^d$ for some $d\in \mathbb{N}$, the Brouwer's degree $\deg(H,\text{int}(Q),\rho e)$ is well defined. Since $h|_{\partial Q} = h_0$ then $H|_{\partial Q} = H_0$ and 
	\begin{equation}\label{s2e2}
	\deg(H,\text{int}(Q), \rho e) = \deg(H_0,\text{int}(Q),\rho e),
	\end{equation}
	in view of the \ Brouwer's degree \ properties. \ Furthermore,  $(iv)$ in $(I_2)$ asserts there exists a unique ${w \in h_0(Q)\cap (\partial B_\rho \cap X)}$ and $\deg(h_0,\text{int}(Q), w) \not = 0$. Hence, for each $u_0 \in Q$ such that $h_0(u_0) = w$, it implies that $H_0(u_0) = \rho e$ and the reverse is also true, then
	\begin{equation}\label{s2e3}
	\deg(H_0,\text{int}(Q), \rho e) = \deg(h_0,\text{int}(Q),w) \not= 0.
	\end{equation}
 Therefore, combining (\ref{s2e2}) and (\ref{s2e3}) we arrive at $\deg(H,\text{int}(Q),\rho e) \not = 0$ and there exists $u_h \in Q$ such 
 that 
 \[
 h(u_h) = \rho e = v_h \in (\partial B_\rho \cap X),
 \] 
which proves the existence of $v_h$. 

\qquad Now we define $\Gamma_0 := \{h_0\}$ and $Q_0 :=\partial Q$, hence from $(ii)$ in $(I_2)$ we have
\begin{equation}\label{s2e4}
\sup_{h_0\in\Gamma_0}\sup_{u\in Q_0}I(h_0(u)) = \sup_{u\in \partial Q}I(h_0(u)) = \omega < \alpha \leq c.
\end{equation}	
In addition, from $(i)$ in $(I_2)$, we conclude that $c \leq \displaystyle\sup_{u\in Q}I(h_0(u))<+\infty$, which along with (\ref{s2e4}) provide condition (\ref{s2e1}) in Proposition \ref{P2.8}. Therefore, applying Proposition \ref{P2.8} and Corollary \ref{C2.9} we guarantee the existence of $(u_n)$, a $(C
)_c$ sequence for $I$, proving the result. 	
\end{proof}	

\section{Setting The Linking Structure}

\qquad In this section we show that functional $I$ defined in (\ref{I}) satisfies $(I_1)-(I_2)$ in Theorem \ref{AR}. In order to do so, we first establish the necessary notation. In fact, we set $X:= E^+$ and $V:=E^-$, since from $(V_2)_0$ we know that $ V$ is finite dimensional. Secondly, in order to distinguish the cases where either $(f_3)_0$ or $(f_3)'_0$ hold, we study autonomous problems and make use of classical results due to Berestycki and Lions \cite{BL} and Jeanjean and Tanaka \cite{JT}. We consider the following equation
\begin{equation}\label{nce1}
-\Delta u = g(u), \quad \text{in} \quad \mathbb{R}^N \quad  N\geq3,
\end{equation}
where it is assumed on $h$ that
\medskip
\newline
$(g_0)$ \ $g:\mathbb{R}\to \mathbb{R}$ is continuous and odd;
\medskip
\newline
$(g_1)$ \ There exists the limit $\displaystyle\lim_{s \to 0^+}\dfrac{g(s)}{s} \in (-\infty,0)$;
\medskip
\newline
$(g_2)$ \ $\displaystyle\lim_{s \to +\infty}|h(s)|s^{-\frac{N+2}{N-2}} = 0$.
\medskip
\newline
Associated to (\ref{nce1}) is the functional $J: H^1(\mathbb{R}^N)\to \mathbb{R}$ given by
\[
J(u) = \dfrac{1}{2}\int_{\mathbb{R}^N}|\nabla u(x)|^2\;dx - \int_{\mathbb{R}^N}G(u(x))\;dx,
\]
where $G(u) = \displaystyle\int_{0}^{u}g(s)\;ds$.

\qquad To study this problem, we recall two remarkable results found in \cite{BL,JT}, we observe that we do not state them in their full generality.

\begin{proposition}[Proposition 4.1 \cite{JT}]\label{P4.1}
	Assume $(g_0)-(g_2)$. Then $J$ is well defined and problem (\ref{nce1}) has a nontrivial solution if and only if $G(s_0)>0$ for some $s_0>0$.
\end{proposition}
\begin{proposition}[Proposition 4.2 \cite{JT}]\label{P4.2}
	Assume $(g_0)-(g_2)$ and $G(s_0)>0$ for some $s_0>0$. Let $\tilde{u}$ be a critical point of (\ref{nce1}) with $\tilde{u}(x)>0$ for all $x\in \mathbb{R}^N$. Then, there exists a path ${\gamma \in C([0,1],H^1(\mathbb{R}^N))}$ such that $\gamma(t)(x)>0$ for all $x\in \mathbb{R}^N$, $t\in (0,1], \ \gamma(0)=0, J(\gamma(1))<0,$  ${\tilde{u} \in \gamma([0,1])}$ and
	\[
	\max_{t\in[0,1]}J(\gamma(t)) = J(\tilde{u}).
	\]
\end{proposition}	

\qquad Returning to problem (\ref{prob_0}), under assumptions of Theorem \ref{main2} we consider the associated ``problem at infinity"
\begin{equation}\label{nce2}
\Delta u + V_{0,\infty}u = f_0(u), \quad \text{in} \quad \mathbb{R}^N, \quad  N\geq3,
\end{equation}
which is equivalent to problem (\ref{nce1}) with
\[
g(s) = \left \{\begin{array}{ll}
-V_{0,\infty}s + f_0(s), \quad \text{for} \ s\geq 0,\\
-g(-s), \quad \quad \quad \quad \text{for} \ s<0,
\end{array}
\right.
\]
since the ground state solution of (\ref{nce2}) is nonnegative. If $I_\infty : H^1(\mathbb{R})\to \mathbb{R}$ is the functional associated to (\ref{nce2}), we observe that applying Proposition \ref{P4.1} to $J = I_\infty$, we conclude that $I_\infty$ has no nontrivial critical points in case $(f_3)_0$ is satisfied.

\qquad In case problem (\ref{nce2}) has no nontrivial solutions, in order to prove that $I$ satisfies $(I_1)-(I_2)$ in Theorem \ref{AR}, we choose $h_0 \equiv Id$, the identity map. Thereby, from hypothesis $(V_1)_0$ one has $a_0> \sigma^+ - f'(0)>0$, hence Theorem 1.1' in \cite{BS} asserts the spectral family of operator $A_0$ ensures the existence of some $e \in X= E^+$ with $||e|| =1$ and satisfying
\begin{equation}\label{le}
\big(\sigma^+ - f'(0)\big)||e||_{L^2(\mathbb{R}^N)}^2\leq ||e||^2 = 1 < a_0||e||^2_{L^2(\mathbb{R}^N)}.
\end{equation}
Choosing such an $e$ and defining $Q$ as in Theorem \ref{AR}, next lemma shows that $I$ satisfies $(I_1)-(I_2)$ for  sufficiently small $\rho>0$, for some $\alpha>0 = \omega$ and for $R>\rho$ large enough.

\begin{lemma} \label{l2}
Assuming $(f_1)_0-(f_2)_0$ and $(V_1)_0-(V_2)_0$, there exist $\alpha >0$ and $R>\rho>0$ such that 
\medskip
\newline
$(i)\quad I|_{\partial B_\rho\cap X} \geq \alpha$;
\medskip
\newline
$(ii)\quad \displaystyle\sup_{u\in Q}I(u) <+\infty;$
\newline
$(iii)\quad \displaystyle\sup_{u\in \partial Q}I(u)= 0$;
\medskip
\newline
$(iv)\quad \partial Q\cap (\partial B_\rho \cap X) = \emptyset$;
\medskip
\newline
$(v)\quad$There exists a unique $w \in Q\cap(\partial B_\rho \cap X)$ and $\deg(Id,\text{int}(Q),w)=1 \not=0$.
\end{lemma}

\begin{proof}
From $(f_1)_0-(f_2)_0$, given $\varepsilon>0$ and $p \in [2,2^*)$ there exists $C_\varepsilon>0$ such that
\[
|F_0(t)|\leq \dfrac{1}{2}\varepsilon|t|^2 + \dfrac{C_\varepsilon}{p}|t|^p \quad \forall \; t \in [0,1],
\]
thus, given $u^+ \in \partial B_\rho \cap X$ one has
\begin{eqnarray*}
I(u^+) &=& \dfrac{1}{2}||u^+||^2 - \int_{\mathbb{R}^N}F_0(u^+(x))\;dx\\
&\geq&\dfrac{1}{2}||u^+||^2 - \dfrac{1}{2}\varepsilon||u^+||^2_{L^2(\mathbb{R}^N)} - \dfrac{C_\varepsilon}{p}||u^+||^p_{L^p(\mathbb{R}^N)}\\
&\geq& \dfrac{r^2}{2}\Big(1 - \varepsilon C^2_2 - \dfrac{C_\varepsilon}{p}C_p^pr^{p-2}\Big)\\
&=&\alpha>0,
\end{eqnarray*}
where, $C_q>0$ is the constant given by the continuous embedding $E\hookrightarrow L^q(\mathbb{R}^N)$ for $q \in [2,2^*]$ and we choose $\varepsilon>0, \; \rho>0$ small enough to guarantee $\alpha>0$. Thus, we have just proved $(i)$.

\qquad On the other hand, to prove $(ii)$, we observe that $I$ is continuous and $Q\subset \mathbb{R}e \oplus V$ is a compact set, hence,
\[
\sup_{u\in Q}I(u) = \max_{u \in Q}I(u)<+\infty.
\]

\qquad To prove $(iii)$, first we note that 
\[
\partial Q = (\partial B_R \cap V)\oplus \{re:0\leq r\leq R\} \cup (\bar{B}_R\cap V) \cup (\bar{B}_R\cap V)\oplus \{Re\}.
\]
Hence, in case $u\in(\partial B_R \cap V)\oplus \{re:0\leq r\leq R\}$, we have
 \[
 I(u) = \dfrac{1}{2} \left[r^2 - R^2\right] - \int_{\mathbb{R}^N}F_0(u)\;dx \leq \dfrac{1}{2} \left[r^2 - R^2\right] \leq0.
 \]
 In case $u \in \bar{B}_R\cap V$, we have
 \[
 I(u) = - \dfrac{1}{2}||u||^2 - \int_{\mathbb{R}^N}F_0(u)\;dx \leq 0.
 \]
 Lastly, in case $u \in (\bar{B}_R\cap V)\oplus \{Re\}$, we can write $u = v+re$ with $v\in \bar{B}_R\cap V$, then
 \begin{eqnarray}\label{s3e1}
  I(u) &=&  \dfrac{1}{2} \left[R^2 - ||v||^2\right] - \int_{\mathbb{R}^N}F_0(Re + v)\;dx\nonumber\\
  & \leq& \dfrac{1}{2}R^2 - \int_{\Omega}F_0(Re + v)\;dx\nonumber\\
  &=& \dfrac{R^2}{2}\left(1 - a_0\int_{\Omega}e^2(x)\;dx\right) - \dfrac{a_0}{2}\int_{\Omega}v^2(x)\;dx + o_R(1),
 \end{eqnarray}
 as $R\to +\infty$, for a arbitrary bounded domain $\Omega \subset \mathbb{R}^N$, provided that $(f_2)_0$ and Lebesgue Dominated Convergence Theorem imply that
 \[
 \lim_{R \to +\infty}\int_{\Omega}\dfrac{F_0(Re(x)+v(x))}{(Re(x)+v(x))^2} (Re(x)+v(x))^2\;dx = \dfrac{a_0}{2}\int_{\Omega}(R^2e^2(x)+v^2(x))\;dx,
 \]
 since $e$ and $v$ are orthogonal. Thus, we arrive at (\ref{s3e1}). Furthermore, from (\ref{le}) we can choose $\Omega$ such that
 \begin{equation}\label{s3e2}
 1 - a_0\int_{\Omega}e^2(x)\;dx <0.
 \end{equation}
Combining (\ref{s3e1}) and (\ref{s3e2}) we get
\begin{equation}\label{s3ee1}
I(u) \leq \dfrac{R^2}{2}\left(1 - a_0\int_{\Omega}e^2(x)\;dx\right) + o_R(1) < 0,
\end{equation}
for $R>\rho>0$ sufficiently large. Therefore $\displaystyle\sup_{u\in \partial Q}I(u) = 0$, finishing the proof of $(iii)$.

\qquad For proving $(iv)$ it is only necessary observe that $R>\rho$, then 
\[
\partial Q\cap (\partial B_\rho \cap X) = (\partial Q \cap \partial B_\rho) \cap X = (V\cap\partial B_\rho)\cap X  = \emptyset.
\]
Moreover, since $R>\rho$,
\[
Q\cap(\partial B_\rho \cap X) = \{re:0\leq r\leq R\}\cap \partial B_\rho = \{\rho e\} \quad \text{and} \quad 
	\rho e \in \text{int}(Q),
\]
we get $(v)$ as consequence of Brouwer's degree properties. Therefore, we have finished the proof.
\end{proof} 

\qquad Note that the strict inequality in (\ref{le}), inherited from $(V_1)_0$, was essential to obtain the suitable function $e$ to prove $(iii)$ in Lemma \ref{l2}. Furthermore, by Lemma \ref{l2} $I$ satisfies all assumptions in Theorem \ref{AR}, hence we are able to apply it, getting a $(C)_c$ sequence for $I$ in case problem (\ref{nce2}) has no nontrivial solutions.

\qquad On the other hand, in case problem (\ref{nce2}) has a solution $\tilde{u}$, by applying Propositions \ref{P4.1} and \ref{P4.2} we can consider a positive solution and provide a path
\[
\gamma(t)(x) := \tilde{u}\left(\dfrac{x}{tL}\right),
\]
for a sufficiently large $L>0$ to be chosen. Details about $\gamma$ can be found in \cite{JT}, in the proof of Proposition 4.2. Since $I_\infty$ is invariant by translations, we can redefine $\gamma$ as 
\begin{equation}\label{s3ee2}
\gamma(t)(x) := \tilde{u}\left(\frac{x-y}{tL}\right),
\end{equation}
for some $y\in \mathbb{R}^N$.

\qquad Now, are going to show that for a suitable $y$ we are able to use $\gamma$ to define $h_0$, in case problem (\ref{nce2}) has a nontrivial solution. First, decomposing $\gamma(t) = \gamma^X(t)+ \gamma^V(t)$, with $\gamma^X(t)\in X= E^+$ and $\gamma^V(t) \in V = E^-$, we claim that, for sufficiently large $|y|>0$ does not depending on $t$, we have $||\gamma^V(t)||$ as small as necessary. In fact, if $\lambda_1<0$ is the smallest eigenvalue of $A_0$ then
\[
||\gamma^V(t)||^2\leq -\lambda_1||\gamma^V(t)||^2_{L^2(\mathbb{R}^N)} \leq -\lambda_1||\gamma(t)||^2_{L^2(\mathbb{R}^N)}= -\lambda_1(tL)^N||\tilde{u}||^2_{L^2(\mathbb{R}^N)}.
\]
Hence, given $\varepsilon>0$ there exists $\delta>0$ independent of $y$ such that
\begin{equation}\label{gammadelta}
||\gamma^V(t)||^2< \varepsilon \quad \text{for} \quad 0<t<\delta.
\end{equation}
Moreover, if $\tilde{u}^V \in E$ is such that $\gamma^V(t)(x) = \tilde{u}^V\left(\dfrac{x-y}{tL}\right) \in V = E^- $ then
\begin{equation}\label{gammaV}
||\gamma^V(t)||^2 = - (tL)^{N-2} \int_{\mathbb{R}^N}\left[|\nabla \tilde{u}^V(x)|^2 + (tL)^2V_0(tLx+y)(\tilde{u}^V(x))^2\right] dx.
\end{equation}
We are going to show that, given $\varepsilon>0$, for sufficiently large $|y|$, we have from (\ref{gammaV}) that
\begin{equation}\label{gamma0}
||\gamma^V(t)||^2 < -(tL)^{N}\int_{\mathbb{R}^N}V_0(tLx+y)(\tilde{u}^V(x))^2\; dx < \varepsilon \quad \text{for} \quad \delta\leq t\leq 1.
\end{equation}
In order to show (\ref{gamma0}), we observe that in view of $(V_2)_0$, given $\varepsilon' >0$, there exists $\tau>0$ such that for $|z|\geq\tau$ it follows that $|V_0(z)-V_{0,\infty}|<\varepsilon'$. Then, concerned to $t\geq \delta$, one has $|tLx|\geq \delta L|x|\geq \tau$, when $|x|\geq \dfrac{\tau}{\delta L}$. Hence, for sufficiently small $\varepsilon'$, and provided that  $V_{0,\infty} > 0$ we arrive at
\begin{eqnarray}\label{s3ee4a}
-(tL)^2\int_{\mathbb{R}^N\setminus B_{\frac{\tau}{\delta L}}}V_0(tLx)(\tilde{u}^V(x-y))^2\; dx &\leq& -\big(V_{0,\infty} -\varepsilon'\big)(\delta L)^2\int_{\mathbb{R}^N\setminus B_{\frac{\tau}{\delta L}}}(\tilde{u}^V(x-y))^2\; dx\nonumber\\
&<& 0.
\end{eqnarray}
In addition, if $|y|\to +\infty$, from the exponential decay of $\tilde{u}$, we have that
\begin{equation}\label{decay}
\int_{B_{\frac{\tau}{\delta L}}}(\tilde{u}(x-y))^2\; dx \to 0.
\end{equation}
Hence, for sufficiently large $|y|$ we get
\begin{eqnarray}\label{s3ee5a}
-(tL)^2\int_{B_{\frac{\tau}{\delta L}}}V_0(tLx)(\tilde{u}^V(x-y))^2\; dx &\leq& -(tL)^2\inf_{z\in\mathbb{R}^N}V_0(z)\int_{B_{\frac{\tau}{\delta L}}}(\tilde{u}^V(x-y))^2\; dx\nonumber\\
&\leq& (t L)^2\big|\inf_{z\in\mathbb{R}^N}V_0(z)\big|\int_{B_{\frac{\tau}{\delta L}}}(\tilde{u}^V(x-y))^2\; dx\nonumber\\
&\leq& L^2 \big|\inf_{z\in\mathbb{R}^N}V_0(z)\big| \int_{B_{\frac{\tau}{\delta L}}}(\tilde{u}(x-y))^2\; dx\nonumber\\ &<& \varepsilon.
\end{eqnarray} 
Finally, combining (\ref{s3ee4a})-(\ref{s3ee5a}) we obtain (\ref{gamma0}). Therefore, $||\gamma^V(t)||^2< \varepsilon$ for all $t\in [0,1]$. 

\qquad Next, we also claim that, for sufficiently large $|y|>0$ not depending on $t$, we have  $\gamma^X(t)\not=0$ for all $t\in [0,1]$. In fact, since 
\[
||\gamma^X(t)||^2 = (A_0\gamma(t),\gamma(t)) + ||\gamma^V(t)||^2,
\]
 it is only  necessary to prove that for each $t\in[0,1]$ one has
\begin{equation*}
(A_0\gamma(t),\gamma(t)) = \int_{\mathbb{R}^N}\left[|\nabla \gamma(t)(x)|^2 + V_0(x)\gamma(t)^2(x)\right] dx > 0.
\end{equation*}
Changing variables, it is equivalent to show that
\begin{equation}\label{s3ee3}
\int_{\mathbb{R}^N}\left[|\nabla \tilde{u}(x)|^2 + (tL)^2V_0(tLx+y)\tilde{u}^2(x)\right] dx >0.
\end{equation}
Since $V_0$ is bounded in $\mathbb{R}^N$, there exists $\delta=\delta(L)>0$ small enough, such that for all $0<t<\delta$, one has 
\begin{eqnarray*}
	\int_{\mathbb{R}^N}|\nabla \tilde{u}(x)|^2\;dx &>& (\delta L)^2|\inf_{z\in\mathbb{R}^N}V_0(z)|\int_{\mathbb{R}^N}\tilde{u}^2(x)\;dx \\
	&>& -(tL)^2\inf_{z\in\mathbb{R}^N}V_0(z)\int_{\mathbb{R}^N}\tilde{u}^2(x)\;dx\\
	&\geq& -(tL)^2\int_{\mathbb{R}^N}V_0(tLx+y)\tilde{u}^2(x)\; dx.
\end{eqnarray*}
Thus, (\ref{s3ee3}) holds for sufficiently small $t$. Furthermore, in view of $(V_2)_0$, concerned to $t\geq \delta$, for sufficiently small $\varepsilon'>0$, we arrive at the same conclusion of (\ref{s3ee4a}). On the other hand, if $|y|\to +\infty$, from the exponential decay of $\tilde{u}$, we obtain (\ref{decay}) and then for sufficiently large $|y|$ we get
\begin{eqnarray}\label{s3ee5}
-(tL)^2\int_{B_{\frac{\tau}{\delta L}}}V(tLx)\tilde{u}^2(x-y)\; dx &\leq& -(tL)^2\inf_{z\in\mathbb{R}^N}V(z)\int_{B_{\frac{\tau}{\delta L}}}\tilde{u}^2(x-y)\; dx\nonumber\\
&\leq& (t L)^2\big|\inf_{z\in\mathbb{R}^N}V(z)\big|\int_{B_{\frac{\tau}{\delta L}}}\tilde{u}^2(x-y)\; dx\nonumber\\
&\leq& L^2 \big|\inf_{z\in\mathbb{R}^N}V(z)\big| \int_{B_{\frac{\tau}{\delta L}}}\tilde{u}^2(x-y)\; dx\nonumber\\ &\leq& \dfrac{1}{2}\int_{\mathbb{R}^N}|\nabla \tilde{u}(x)|^2\;dx.
\end{eqnarray}
Finally, combining (\ref{s3ee4a}) and (\ref{s3ee5}) we obtain
\[
\int_{\mathbb{R}^N}|\nabla \tilde{u}(x)|^2\;dx > -(tL)^2\int_{B_{\frac{\tau}{\delta L}}}V(tLx)\tilde{u}^2(x-y)\; dx > -(tL)^2\int_{\mathbb{R}^N}V(tLx)\tilde{u}^2(x-y)\; dx,
\]
proving (\ref{s3ee3}) also for $\delta\leq t\leq 1$. Therefore, $\gamma^X(t) \in X = E^+$ is not null, for all $t\in [0,1]$. 

\qquad In order to prove that I satisfies $(I_1)-(I_2)$ in Theorem 1.2, when problem (\ref{nce2}) has a solution, we choose
\begin{equation}\label{h_0}
h_0(u) := \gamma^X(t) + |v|, \quad \text{for \ each} \quad u = Rte + v \in Q, \; t \in [0,1],  
\end{equation}
where $|v|(x) = |v(x)|$ is the modulus of function $v$. Moreover, we note that if $v \in V$, since $v = v^+ - v^-$ and $|v|=v^+ + v^-$, where $v^+=\max\{v,0\} \in V$ and $v^-=\max\{0,-v\} \in V$, we conclude that $|v|\in V$.
Thus, choosing $e$ and $Q$ as before, the following lemma gives $(I_1)-(I_2)$ for $I$, in case problem (\ref{nce2}) has a solution.

\begin{lemma} \label{l3}
	Assuming $(f_1)_0-(f_2)_0$ and $(V_1)_0-(V_3)_0$, there exist $\alpha >0$ and $R>\rho>0$ such that 
	\medskip
	\newline
	$(i)\quad I|_{\partial B_\rho\cap X} \geq \alpha$;
	\medskip
	\newline
	$(ii)\quad \displaystyle\sup_{u\in Q}I(h_0(u)) <+\infty;$
	\newline
	$(iii)\quad\displaystyle\sup_{u\in \partial Q}I(h_0(u))= 0$;
	\medskip
	\newline
	$(iv)\quad h_0(\partial Q)\cap (\partial B_\rho \cap X) = \emptyset$;
	\medskip
	\newline	
$(v)\quad$There exists a unique $w \in h_0(Q)\cap(\partial B_\rho \cap X)$ and $\deg(h_0,\text{int}(Q),w)\not=0$.
\end{lemma}

\begin{proof} Since the proof of $(i)$ does not depend on $h_0$, it is exactly the same as in Lemma \ref{l2}. In order to prove $(ii)$, we also observe that $I\circ h_0$ is continuous and $Q\subset \mathbb{R}e\oplus V$ is a compact set, hence
	\[
	\sup_{u\in Q}I(h_0(u)) = \max_{u \in Q}I(h_0(u)) <+\infty.
	\]
	
\qquad Now, with the purpose of proving $(iii)$, we recall that
	\[
	\partial Q = (\partial B_R \cap V)\oplus \{re:0\leq r\leq R\} \cup (\bar{B}_R\cap V) \cup (\bar{B}_R\cap V)\oplus \{Re\},
	\]
hence, if $u = v + Rte\in(\partial B_R \cap V)\oplus \{re:0\leq r\leq R\}$ with $v \in\partial B_R \cap V$, we have
\begin{eqnarray*}
I(h_0(u)) &=& \dfrac{1}{2} \left[||\gamma^X(t)||^2 - R^2\right] - \int_{\mathbb{R}^N}F_0(\gamma^X(t) + |v|)\;dx\\
 &\leq& \dfrac{1}{2} \left[\max_{t \in [0,1]}||\gamma^X(t)||^2 - R^2\right]\nonumber\\
  &\leq& \dfrac{1}{2} \left[\max_{t \in [0,1]}||\gamma(t)||^2 - R^2\right] <0,
\end{eqnarray*}
for $R>0$ large enough.
If $u = v \in \bar{B}_R\cap V$, we have
\[
I(h_0(v)) = - \dfrac{1}{2}||v||^2 - \int_{\mathbb{R}^N}F_0(|v|)\;dx \leq 0.
\]
Lastly, if $u = v +Rte \in (\bar{B}_R\cap V)\oplus \{Re\}$, with $v\in \bar{B}_R\cap V$, then
\begin{eqnarray}\label{s3e3}
I(h_0(u)) &=&  \dfrac{1}{2} \left[||\gamma^X(1)||^2 - ||v||^2\right] - \int_{\mathbb{R}^N}F_0(\gamma^X(1) + |v|)\;dx\nonumber\\
& \leq& I(\gamma^X(1)) + \int_{\mathbb{R}^N}F_0(\gamma^X(1))\;dx - \int_{\mathbb{R}^N}F_0(\gamma^X(1) + |v|)\;dx \nonumber\\
&\leq& -\int_{\mathbb{R}^N}\Big[F_0(\gamma(1) + |v|) - F_0(\gamma(1))\Big]dx \nonumber\\
&\leq&0.
\end{eqnarray}
In fact, since $I_\infty(\gamma(1))<0$ and in view of $(V_3)_0$ we have $I\leq I_\infty$, then $I(\gamma(1)) < 0$. Moreover, for sufficient large $|y|>0$ we get $||\gamma^V(1)||$ small enough such that $I(\gamma(1))<0$ implies $I(\gamma^X(1))\leq 0$. In addition, provided that $f_0(s)\geq 0$ for $s\geq0$, then by the definition of $F_0$, we have $F_0(\gamma(1)+|v|)\geq F_0(\gamma(1))$, in view of $|v|\geq0$. Therefore, $\displaystyle\sup_{u\in \partial Q}I(h_0(u)) = 0$.

\qquad For obtaining $(iv)$ first we note that
\[
h_0(\partial Q) = (\partial B_R \cap V)\oplus \{\gamma^X(t):0\leq t\leq 1\} \cup (\bar{B}_R\cap V) \cup (\bar{B}_R\cap V)\oplus \{\gamma^X(1)\}
\]
and that 
\[
(\partial B_R \cap V)\oplus \{\gamma(t):0\leq t\leq 1\}\cap X = \emptyset.
\]
In addition, to guarantee that $(\bar{B}_R\cap V)\oplus \{\gamma^X(1)\} \cap X\cap \partial B_\rho = \emptyset$ it is enough to choose a sufficiently large $L>0$ such that 
\begin{equation}\label{s3e4}
||\gamma^X(1)||^2 = L^{N-2}\int_{\mathbb{R}^N}\left[|\nabla \tilde{u}^X(x)|^2 + (L)^2V(Lx+y)(\tilde{u}^X(x))^2\right] dx> \rho^2,
\end{equation}
where $\gamma^X(1) = \tilde{u}^X\left(\dfrac{x-y}{L}\right)$. Then, we conclude that
\[
h_0(\partial Q)\cap (\partial B_\rho \cap X) = h_0(\partial Q)\cap \partial B_\rho\cap X = (\bar{B}_R\cap V) \cap \partial B_\rho\cap X = \emptyset.
\]

\qquad Finally, the function $\psi: [0,1]\to \mathbb{R}$, given by $\psi(t) = ||\gamma^X(t)||$, it strictly increasing and hence injective. Moreover, $\psi$ is continuous, $\psi(0)=0$ and from (\ref{s3e4}) we have $\psi(1)>\rho$. Thus, from the Intermediate Value Theorem there exists some (unique, since $\psi$ is injective) $t_0 \in (0,1)$ such that $\psi(t_0) = \rho.$ Hence,
\[
h_0(Q)\cap(\partial B_\rho \cap X)  = \{\gamma^X(t): t \in [0,1]\}\cap \partial B_\rho = \{\gamma^X(t_0)\},
\]
\medskip
and there exists a unique $w = \gamma^X(t_0)\in h_0(Q)\cap(\partial B_\rho \cap X)$. Since $Rte\mapsto h_0(Rte) = \gamma^X(t)$ is injective, there exists a unique $u_0 = Rt_0e\in \text{int}(Q)$ such that $h_0(u_0) = \gamma^X(t_0)$. Therefore,
$\deg(h_0,\text{int}(Q),w) \not=0$, proving $(v)$.
\end{proof}

\qquad Observe that different from Lemma \ref{l2}, the strict inequality in (\ref{le}) was not essential to obtain $(iii)$ in Lemma \ref{l3}. In fact, now the essential were the properties satisfied by $\gamma(t)$. Provided that $I$ satisfies all assumptions in Theorem \ref{AR}, we are able to apply it again, getting a $(C)_c$ sequence for $I$ also when problem (\ref{nce2}) has a solution. Finally, we point out that since the linking structure changes according to either problem (\ref{nce2}) has a solution or not, then the linking level $c$ may be different in each case.

\section{Boundedness of Cerami Sequences}

\qquad In this section we are going to prove that, under the assumptions of Theorem \ref{main2}, every Cerami sequence for $I$ is bounded, particularly, those found by Theorem \ref{AR} in the previous section. 

\begin{lemma}\label{B2}
	Let $(u_n)$ be a $(C)_c$ sequence for $I$. Under the assumption of $(V_1)_0-(V_2)_0$, $(f_1)_0-(f_2)_0$ and either $(f_3)_0$ or $(f_3)'_0$, it follows that $(u_n)$ is bounded.
\end{lemma}

\begin{proof}
Arguing by contradiction we suppose $||u_n||\to +\infty$
 as $n\to +\infty$ up to subsequences. Defining $v_n:= \dfrac{u_n}{||u_n||}$, we have $(v_n)\subset E$ bounded, hence $v_n\to v$ in $E$ and $v_n\to v$ in $L^q_{loc}(\mathbb{R}^N)$ for $q \in [2,2^*)$, then $v_n(x)\to v(x)$ almost everywhere in $\mathbb{R}^N$. Furthermore, since $(v_n)$ is bounded it must satisfy either
 \bigskip
 \newline
  $(1)$ {\bf Vanishing:} For all $R>0$, 
 \[
 \displaystyle\lim_{n \to +\infty}\sup_{y \in \mathbb{R}^N}\int_{B_r(y)}v_n^2(x)\;dx = 0;
 \]
 or
 \medskip
 \newline
 $(2)$ {\bf Non-vanishing:} There exist $\eta>0$, $0<R<+\infty$ and $(y_n)\subset \mathbb{R}^N$ such that 
 \[
 \displaystyle\lim_{n \to +\infty}\int_{B_r(y_n)}v_n^2(x)\;dx \geq \eta >0.
 \]
 We prove that both are impossible for $(v_n)$ arriving at the desired contradiction.
 
 \qquad First we prove that $(1)$ is not possible to $(v_n)$. In fact, supposing that $(v_n)$ is a vanishing sequence, in view of $(V_2)_0$ one has
 \begin{equation}\label{lb21}
 \int_{\mathbb{R}^N}\big(V_0(x)-V_{0,\infty}\big)v^2_n(x)\;dx \to 0 \quad \text{as} \quad n\to+\infty.
 \end{equation}
 Then, in view of (\ref{lb21}) we arrive at
 \[
 \displaystyle\lim_{n \to +\infty}\int_{\mathbb{R}^N}\left(|\nabla v_n(x)|^2+V_{0,\infty}v^2_n(x)\right)dx = \lim_{n \to +\infty}\int_{\mathbb{R}^N}\left(|\nabla v_n(x)|^2+V_0(x)v^2_n(x)\right)dx.
 \]
 Hence,
 \begin{eqnarray*}
 \displaystyle\limsup_{n \to +\infty}\int_{\mathbb{R}^N}V_{0,\infty}v^2_n(x)\;dx &\leq&
 \displaystyle\lim_{n \to +\infty}\int_{\mathbb{R}^N}\left(|\nabla v_n(x)|^2+V_{0,\infty}v^2_n(x)\right)dx\\
 &=& \lim_{n \to +\infty}\left(||v^+_n||^2-||v^-_n||^2\right)\\
 &\leq&\lim_{n \to +\infty}||v_n||^2=1
 \end{eqnarray*}
and
\begin{equation}\label{lb22}
\displaystyle\limsup_{n \to +\infty}\int_{\mathbb{R}^N}v^2_n(x)\;dx \leq \dfrac{1}{V_{0,\infty}}.
\end{equation}
Furthermore, since $(u_n)$ is a Cerami sequence, it follows that
\begin{equation*}
	o_n(1) = \dfrac{I'(u_n)}{||u_n||}\left(v^+_n - v^-_n\right) = 1 - \int_{\mathbb{R}^N}\dfrac{f_0(u_n(x))}{u_n(x)}\left[(v_n^+(x))^2 - (v_n^-(x))^2\right]dx,
\end{equation*}
then
\begin{equation}\label{lb23}
\lim_{n \to +\infty}\int_{\mathbb{R}^N}\dfrac{f_0(u_n(x))}{u_n(x)}\left[(v_n^+(x))^2-(v_n^-(x))^2\right]dx = 1.
\end{equation}

\qquad Now, if $(f_3)_0$ is satisfied, from (\ref{lb22}) it implies that
\begin{eqnarray}\label{lb24}
1 &=& \lim_{n \to +\infty}\int_{\mathbb{R}^N}\dfrac{f_0(u_n(x))}{u_n(x)}\left[(v_n^+(x))^2-(v_n^-(x))^2\right]dx\nonumber\\
&\leq& \limsup_{n \to +\infty}\int_{\mathbb{R}^N}\left(V_{0,\infty} - \delta\right)\left[(v_n^+(x))^2-(v_n^-(x))^2\right]dx\nonumber\\
&\leq&\left(V_{0,\infty} - \delta\right)\limsup_{n \to +\infty}\int_{\mathbb{R}^N}v^2_n(x)\;dx\nonumber\\
&\leq&1 - \dfrac{\delta}{V_{0,\infty}},
\end{eqnarray}
which is an absurd. On the other hand, if $(f_3)'_0$ holds true, we define
\[
\Omega_n:= \left\{x\in \mathbb{R}^N: \dfrac{f_0(u_n(x))}{u_n(x)} \leq V_{0,\infty}\right\},
\]
and again from (\ref{lb22}), we conclude that 
\begin{equation}\label{lb25}
\int_{\Omega_n}\dfrac{f_0(u_n(x))}{u_n(x)}\left[(v^+_n(x))^2 - (v^-_n(x))^2\right]dx\leq \left(V_{0,\infty}-\delta\right)\int_{\Omega_n}v_n^2(x)\;dx \leq 1 - \dfrac{\delta}{V_{0,\infty}}.
\end{equation}
Thereby, in virtue of (\ref{lb23}) and (\ref{lb25}) we arrive at
\begin{equation}\label{lb26}
\liminf_{n \to +\infty}\int_{\mathbb{R}^N\setminus\Omega_n}\dfrac{f_0(u_n(x))}{u_n(x)}\left[(v^+_n(x))^2 - (v^-_n(x))^2\right]dx \geq \dfrac{\delta}{V_{0,\infty}} >0.
\end{equation}

\qquad Since ${f(s)}/{s}$ is bounded, it follows that
\begin{equation}\label{lb27}
0< \liminf_{n \to +\infty}\int_{\mathbb{R}^N\setminus\Omega_n}\dfrac{f_0(u_n(x))}{u_n(x)}\left[(v^+_n(x))^2 - (v^-_n(x))^2\right]dx\leq C\liminf_{n \to +\infty} \int_{\mathbb{R}^N\setminus\Omega_n}v^2_n(x)\;dx,
\end{equation}
then applying H\"older inequality with some $q \in (2,2^*)$, we obtain
\begin{equation}\label{lb28}
\liminf_{n \to +\infty} \int_{\mathbb{R}^N\setminus\Omega_n}v^2_n(x)\;dx\leq |\mathbb{R}^N\setminus \Omega_n|^{\frac{2^*}{2^*-q}}||v_n||^2_{L^{2\frac{2^*}{q}}(\mathbb{R}^N)},
\end{equation}
and since $(v_n)$ is a vanishing sequence, by Lions' Lemma,  $||v_n||^2_{L^{2\frac{2^*}{q}}(\mathbb{R}^N)} \to 0$ as $n \to +\infty$, hence from (\ref{lb28}), we conclude that $|\mathbb{R}^N\setminus \Omega_n| \to +\infty$ as $n \to+\infty$ so that 
\begin{equation}\label{lb29}
\int_{\mathbb{R}^N}\hat{F}_0(u_n(x))\;dx\geq \int_{\mathbb{R}^N\setminus \Omega_n}\hat{F}_0(u_n(x))\;dx\geq \delta|\mathbb{R}^N\setminus \Omega_n|
\end{equation}
and (\ref{lb29}) implies that 
\begin{equation}\label{lb210}
\displaystyle\lim_{n \to +\infty}\int_{\mathbb{R}^N}\hat{F}_0(u_n(x))\;dx = +\infty.
\end{equation}
On the other hand,
\begin{equation}\label{lb211}
\int_{\mathbb{R}^N}\hat{F}_0(u_n(x))\;dx = 2I(u_n) - I'(u_n)u_n \to 2c \quad \text{as} \quad n \to +\infty.
\end{equation}
Therefore, (\ref{lb211}) contradicts (\ref{lb210}) and $(v_n)$ does not satisfy $(1)$.

\qquad Henceforth we show that $(2)$ is also impossible for $(v_n)$. Supposing that $(v_n)$ is non-vanishing we consider two cases:
\medskip
\newline
{\bf Case i:} $(y_n)$ is bounded. Since $(y_n)$ is bounded, then $y_n \to y$, up to subsequences, then considering $\tilde{v}_n(x):= v_n(x+y_n)$ we have $(\tilde{v}_n)$ bounded from the equivalence of norms and then $\tilde{v}_n \rightharpoonup \tilde{v}$ in $E$, up to subsequences. Analogously to $(v_n)$ we have $\tilde{v}_n(x) \to \tilde{v}(x)$ almost everywhere, however, $v_n(x+y_n) \to v(x+y)$ almost everywhere, hence $\tilde{v}(x) = v(x+y)$ almost everywhere in $\mathbb{R}^N$ and $v\equiv 0$ if and only if $\tilde{v}\equiv 0$. Provided that $(v_n)$ satisfies $(2)$ we have
 \begin{eqnarray}\label{nvc}
||\tilde{v}||^2_{L^2(\mathbb{R}^N)}&\geq&||\tilde{v}||^2_{L^2(B_R(0))}\nonumber\\
&=&\displaystyle\lim_{n \to +\infty}\int_{B_r(0)}\tilde{v}_n^2(x)\;dx\nonumber\\
&=&  \displaystyle\lim_{n \to +\infty}\int_{B_r(0)}v_n^2(x+y_n)\;dx\nonumber\\
&=& \displaystyle\lim_{n \to +\infty}\int_{B_r(y_n)}v_n^2(x)\;dx\nonumber\\
&\geq& \eta >0,
\end{eqnarray}
proving that $\tilde{v} \not= 0$, therefore $v\not = 0.$

\qquad Moreover, we claim $v$ is an eigenvector of $A_0$ associated to $a_0$. In fact, given $\varphi \in C^\infty_0(\mathbb{R}^N)$, since $(u_n)$ is a Cerami sequence
\begin{equation}\label{lb213}
\int_{\mathbb{R}^N}\nabla v_n(x) \cdot \nabla \varphi(x) \;dx + \int_{\mathbb{R}^N}V_0(x)v_n(x)\varphi(x)\;dx = \int_{\mathbb{R}^N}\dfrac{f_0(u_n(x))}{u_n(x)}v_n(x)\varphi(x)\;dx + o_n(1).
\end{equation}
By the weak convergence of $(v_n)$ and from hypothesis $(f_2)_0$, applying Lebesgue Dominated Convergence Theorem, passing (\ref{lb213}) to the limit we obtain
\[
\int_{\mathbb{R}^N}\nabla v(x) \cdot \nabla \varphi(x) \;dx + \int_{\mathbb{R}^N}V_0(x)v(x)\varphi(x)\;dx = a_0\int_{\mathbb{R}^N}v(x)\varphi(x)\;dx,
\]
showing the claim.

\qquad Finally, to arrive at a contradiction, we observe that since $a_0 \notin \sigma_p(A_0)$, then $v$ cannot be an eigenfunction of $A_0$ associated to $a_0$.
Therefore we have proved that $(2)$ cannot occur if $(y_n)$ is bounded, concluding Case i.
\medskip
\newline
{\bf Case ii:} $(y_n)$ is unbounded. Considering again $\tilde{u}_n(x)= u_n(x+y_n)$ and $\tilde{v}_n(x) = v_n(x+y_n)$ we are going to show that for all $\varphi \in C^\infty_0(\mathbb{R}^N)$ one has
\begin{equation}\label{lb215}
\int_{\mathbb{R}^N} \nabla \tilde{v}(x) \cdot \nabla \varphi(x)\;dx + \int_{\mathbb{R}^N}V_{0,\infty}\tilde{v}(x)\varphi(x)\;dx = a_0 \int_{\mathbb{R}^N}\tilde{v}(x) \varphi (x)\;dx,
\end{equation}
with $\tilde{v}\not= 0$ the weak limit of $(\tilde{v}_n)$, which is nonzero in view of (\ref{nvc}), provided that $(v_n)$ is non-vanishing. Then, (\ref{lb215}) implies that $\tilde{v}$ is an eigenvector of $-\Delta + V_{0,\infty}$, with eigenvalue $a_0$, which is a contradiction, since operator $-\Delta$ has no eigenvalues in $\mathbb{R}^N$.

\qquad  In order to prove (\ref{lb215}), given $\varphi \in C^\infty_0(\mathbb{R}^N)$ we define $\varphi_n(x):=\varphi(x-y_n)$ for all $x \in \mathbb{R}^N$, hence $(\varphi_n)$ is bounded in $E$, in view of the equivalence of norms. Since $|y_n|\to +\infty$, up to subsequences, making the necessary change of variables, using the weak convergence information and applying Lebesgue Dominated Convergence Theorem, we obtain
{\small
\begin{eqnarray}\label{lb216}
\dfrac{I'(u_n)}{||u_n||}\varphi_n + \int_{\mathbb{R}^N}\dfrac{f(u_n(x))}{u_n(x)}{v}_n(x)\varphi_n(x)\;dx
&=& \int_{\mathbb{R}^N}\nabla v_n(x)\cdot \nabla \varphi_n(x) \;dx + \int_{\mathbb{R}^N}V_0(x)v_n(x)\varphi_n(x)\;dx\nonumber\\
&=& \int_{\mathbb{R}^N}\nabla \tilde{v}_n(x)\cdot \nabla \varphi(x) \;dx + \int_{\mathbb{R}^N}V_0(x+y_n)\tilde{v}_n(x)\varphi(x)\;dx\nonumber\\
&=& \int_{\mathbb{R}^N}\nabla \tilde{v}(x)\cdot \nabla \varphi(x) \;dx + \int_{\mathbb{R}^N}V_{0,\infty}\tilde{v}(x)\varphi(x)\;dx + o_n(1).
\end{eqnarray}
}
On the other hand, in view of $(f_2)_0$, provided that $(u_n)$ is a Cerami sequence, it follows that
\begin{eqnarray}\label{lb217}
\dfrac{I'(u_n)}{||u_n||}\varphi_n + \int_{\mathbb{R}^N}\dfrac{f(u_n(x))}{u_n(x)}{v}_n(x)\varphi_n(x)\;dx &=&
o_n(1) + \int_{\mathbb{R}^N}\dfrac{f(\tilde{u}_n(x))}{\tilde{u}_n(x)}\tilde{v}_n(x)\varphi(x)\;dx\nonumber\\ 
&=& o_n(1) + a_0\int_{\mathbb{R}^N}\tilde{v}(x)\varphi(x)\;dx.
\end{eqnarray}
Combining (\ref{lb216}) and $(\ref{lb217})$ we arrive at (\ref{lb215}). 
\end{proof}
	 
\section{Proof of Main Results}

\qquad Previous sections have proved the existence of $(u_n)$ a bounded $(C)_c$ sequence for $I$ in both cases, when either $(f_3)$ or $(f_3)'$ is satisfied. Now, by analyzing the existence of solution to problem (\ref{nce2}) we are going to be able to prove Theorem \ref{main2}.

\begin{proof}[Proof of Theorem \ref{main2}]
	Provided that we have the existence of $(u_n)$ a $(C)_c$ bounded sequence for $I$, it implies that $u_n \rightharpoonup u$ in $E$ and $u_n \to u$ in $L^2_{loc}(\mathbb{R}^N)$. Given $\varphi \in C^\infty_0(\mathbb{R}^N)$, let $K\subset\mathbb{R}^N$ be the compact support of $\varphi$. Since $(u_n)$ is a $(C)_c$ sequence for $I$, from the weak convergence and the Lebesgue Dominated Convergence Theorem, it follows that
	\begin{eqnarray}\label{nce3}
	o_n(1) &=& I'(u_n)\varphi\nonumber\\
	&=& \Big(u_n^+ - u_n^-,\varphi\Big) - \int_{\mathbb{R}^N}f_0(u_n(x))\varphi(x)\;dx\nonumber\\ 
	&=& \Big(u^+ - u^-,\varphi\Big) - \int_{K}f_0(u(x))\varphi(x)\;dx + o_n(1)\nonumber\\
	&=& I'(u)\varphi + o_n(1).
	\end{eqnarray}
	Thus, $I'(u)\varphi =0$, and since $\varphi$ is arbitrary, by density we obtain $I'(u)\equiv0$ and therefore $u$ is a critical point of $I$. 
	If $u\not=0$, we obtain a nontrivial critical point of $I$. Then, we suppose by contradiction that $u\equiv 0$ and we are going to arrive at an absurd.
	
	\qquad First, we claim that in this case $(u_n)$ is also a $(C)_c$ sequence for $I_\infty: H^1(\mathbb{R}^N)\to \mathbb{R}$, the functional associated to (\ref{nce2}). In fact,
	\begin{eqnarray}\label{nce4}
	I_\infty(u_n) &=& \dfrac{1}{2}\int_{\mathbb{R}^N}\Big(|\nabla u_n(x)|^2+V_{0,\infty}u_n^2(x)\Big)\;dx - \int_{\mathbb{R}^N}F_0(u_n(x))\;dx\nonumber\\
	&=& I(u_n) - \int_{\mathbb{R}^N}\big(V_{0,\infty} - V_0(x)\big)u^2_n(x)\;dx\nonumber\\
	&=& I(u_n) + o_n(1),
	\end{eqnarray}
	in view of $(V_2)_0$ and since $u_n\to 0$ in $L^2_{loc}(\mathbb{R}^N)$. Moreover, for the same reasons, we obtain
	\[
	\sup_{||v||\leq 1}\left|\Big(I_\infty'(u_n)-{I}'({u}_n)\Big)v\right| = \sup_{||v||\leq 1}\left|\int_{\mathbb{R}^N}\Big(V_{0,\infty} - V_0(x)\Big){u}_n(x)v(x)\;dx\right| \to 0.
	\]
	Now, since $c>0$ we claim that $(u_n)$ does not vanish. In fact, given $\varepsilon>0$ from $(f_1)_0-(f_2)_0$ there exists $C_\varepsilon>0$ such that
	\[
	\int_{\mathbb{R}^N}|f_0(u(x))u(x)|\;dx \leq \varepsilon||u||^2_{L^2(\mathbb{R}^N)} + C_\varepsilon||u||^p_{L^p(\mathbb{R}^N)},
	\]
	thus, if $(u_n)$ vanishes, since $\varepsilon$ is arbitrary, we get that 
	\begin{equation}\label{nce5}
		\displaystyle\int_{\mathbb{R}^N}f_0(u_n(x))\big(u_n^+(x)-u_n^-(x)\big)\;dx \to 0, \quad \text{as} \quad n\to +\infty.
	\end{equation}
Since $(u_n)$ is a $(C)_c$ sequence for $I$ and in virtue of (\ref{nce5}) we arrive at
\begin{equation}\label{nce6}
o_n(1) = I'(u_n)\big(u_n^+ - u_n^-\big)  + \int_{\mathbb{R}^N}f_0(u_n(x))\big(u_n^+(x) - u_n^-(x)\big)\;dx = ||u_n||^2,
\end{equation} 
contradicting that 
\[\displaystyle\liminf_{n \to +\infty}||u_n||^2\geq \displaystyle\liminf_{n \to +\infty}||u_n^+||^2\geq \displaystyle\lim_{n \to +\infty}2I(u_n) = c>0.
\]

\qquad Thus, $(u_n)$ is a non-vanishing sequence, hence there exist $\eta>0$, $R>0$ and $(y_n)\subset \mathbb{R}^N$ such that
\[
\displaystyle\int_{B_R(y_n)}u^2_n(x)\;dx\geq \eta >0.
\]
Setting $\tilde{u}_n(x):=u_n(x+y_n)$ as before, since we have proved $(u_n)$ is a $(C)_c$ sequence for $I_\infty$, then so does $(\tilde{u}_n)$, provided that $I_\infty$ is invariant to translations. Moreover, $(\tilde{u}_n)$ is bounded, hence $\tilde{u}_n \rightharpoonup \tilde{u}$ in $E$ and $\tilde{u}_n \to \tilde{u}$, in $L^2_{loc}(\mathbb{R}^N)$, up to subsequences. Given $\varphi \in C^\infty_0(\mathbb{R}^N)$, let $K\subset\mathbb{R}^N$ be the compact support of $\varphi$. Since $(\tilde{u}_n)$ is a $(C)_c$ sequence for $I_\infty$, from the weak convergence and the Lebesgue Dominated Convergence Theorem, it follows that
\begin{eqnarray}\label{nce7}
o_n(1) &=& I_\infty'(\tilde{u}_n)\varphi\nonumber\\
&=& \int_{\mathbb{R}^N}\Big(\nabla \tilde{u}_n(x)\cdot \nabla \varphi(x) + V_{0,\infty}\tilde{u}_n(x)\varphi(x)\Big)\;dx - \int_{\mathbb{R}^N}f_0(u_n(x))\varphi(x)\;dx\nonumber\\ 
&=& \int_{\mathbb{R}^N}\Big(\nabla \tilde{u}(x)\cdot \nabla \varphi(x) + V_{0,\infty}\tilde{u}(x)\varphi(x)\Big)\;dx - \int_{K}f_0(u(x))\varphi(x)\;dx + o_n(1)\nonumber\\
&=& I_\infty'(\tilde{u})\varphi + o_n(1).
\end{eqnarray}
Since $\varphi$ is arbitrary, in view of (\ref{nce7}) we conclude that $I'_\infty(\tilde{u})\equiv0$. In addition, since $(u_n)$ is non-vanishing arguing as in (\ref{nvc}) with $(u_n)$ and $(\tilde{u}_n)$ instead of $(v_n)$ and $(\tilde{v}_n)$, we get that $\tilde{u}\not= 0$, namely $\tilde{u}$ is a nontrivial critical point to $I_\infty$.

\qquad At this point, if problem (\ref{nce2}) has no a nontrivial solution, in particular when $(f_3)_0$ is satisfied, in view of Proposition \ref{P4.1} we get a contradiction, since $\tilde{u}\not= 0$. Therefore, $u\not=0$ is a nontrivial critical point for $I$. On the other hand, in case problem (\ref{nce2}) has a nontrivial weak solution, if $(f_3)'_0$ and $(V_3)_0$ hold, we have $Q_0(s)\geq0$ and  from Fatou's Lemma we get that
\begin{eqnarray}\label{nce8}
c &=& \lim_{n \to +\infty}\left[I_\infty(\tilde{u}_n) - \dfrac{1}{2}I'_\infty(\tilde{u}_n)\tilde{u}_n\right]\nonumber\\
&=& \lim_{n \to +\infty}\dfrac{1}{2}\int_{\mathbb{R}^N}Q_0(\tilde{u}_n(x))\;dx\nonumber\\
&\geq&  \dfrac{1}{2}\int_{\mathbb{R}^N}Q_0(\tilde{u}(x))\;dx\nonumber\\
&=& I_\infty(\tilde{u}) - \dfrac{1}{2}I'_\infty(\tilde{u})\tilde{u}\nonumber\\
&=& I_\infty(\tilde{u}).
\end{eqnarray}
Thus, $\tilde{u}\not=0$ is a critical point of $I_\infty$ with $I_\infty(\tilde{u})\leq c.$ 
Now, since $h_0$ defined in (\ref{h_0}) belongs to $\Gamma$, we obtain
\begin{equation}\label{s4e1}
c\leq \sup_{u\in Q}I(h_0(u)).
\end{equation}
In addition, since $\gamma^X(t) \in X = E^+$ for all $t\in[0,1]$, and provided that $|v|\geq 0$, by the definition of $F_0(s)$, for all $u = tRe+ v \in Q$  we obtain
\begin{eqnarray*}
I({h}_0(u)) &=& \dfrac{1}{2}\left[||\gamma^X(t)||^2 - ||v||^2 \right] - \int_{\mathbb{R}^N}F_0(\gamma^X(t)+|v|)\;dx \\
&\leq& I(\gamma^X(t)) - \int_{\mathbb{R}^N}\left[F_0(\gamma^X(t)+|v|) - F_0(\gamma^X(t))\right]dx\\
&\leq& I(\gamma^X(t)),
\end{eqnarray*}
hence
\begin{equation}\label{s4e2}
\sup_{u\in Q}I(h_0(u))\leq\sup_{t\in[0,1]}I(\gamma^X(t)).
\end{equation} 
Moreover, if we assume $(V_3)_0$ and also $V_0(x)\not\equiv V_{0,\infty}$, we have
\begin{equation}\label{s4e3}
I(v)<I_\infty(v) \quad \text{for\; all} \quad v \in E.
\end{equation}
Furthermore, we recall that for $|y|>0$ large enough, we get $||\gamma^V(t)||_{L^2(\mathbb{R}^N)}$ sufficiently small as in (\ref{gammadelta}) and (\ref{gamma0}), hence considering that $\gamma^X(t)$ and $\gamma^V(t)$ are orthogonal in $E$, for $\varepsilon>0$ small enough we arrive at
{\small
\begin{eqnarray}\label{s4e4}
I_\infty(\gamma(t)) &=& I_\infty(\gamma^X(t)) + \int_{\mathbb{R}^N}\left[\nabla\gamma^X(t)\cdot \nabla\gamma^V(t) + V_{0,\infty}\gamma^X(t)\gamma^V(t)\right] dx\nonumber\\
&+& \dfrac{1}{2}\int_{\mathbb{R}^N}\left[|\nabla\gamma^V(t)|^2 + V_{0,\infty}|\gamma^V(t)|^2\right] dx - \int_{\mathbb{R}^N}\left[F_0(\gamma(t)) - F_0(\gamma^X(t))\right]dx\nonumber\\
&\geq& I_\infty(\gamma^X(t)) + \int_{\mathbb{R}^N}\left[V_{0,\infty} -V_0(x)\right]\gamma^X(t)\gamma^V(t)\; dx + C||\gamma^V(t)||_{L^2(\mathbb{R}^N))}^2\nonumber\\
&-& \int_{\mathbb{R}^N}f_0(\gamma^X(t) + \theta_t\gamma^V(t))\gamma^V(t)\;dx\nonumber\\
&\geq& I_\infty(\gamma^X(t)) -\varepsilon\int_{\mathbb{R}^N\backslash B_{\tau}}|\gamma^X(t)||\gamma^V(t)|\; dx -C\int_{B_\tau}|\gamma^X(t)||\gamma^V(t)|\;dx +C||\gamma^V(t)||_{L^2(\mathbb{R}^N))}^2\nonumber\\
&-& \int_{\mathbb{R}^N}\left[\varepsilon|\gamma^X(t)||\gamma^V(t)| + \varepsilon|\gamma^V(t)|^2 +C_\varepsilon|\gamma^X(t)|^{p-1}|\gamma^V(t)| +C_\varepsilon|\gamma^V(t)|^p \right]dx\nonumber\\
&\geq& 
I_\infty(\gamma^X(t)) -2\varepsilon||\gamma^X(t)||_{L^2(\mathbb{R}^N)}||\gamma^V(t)||_{L^2(\mathbb{R}^N)} -C\varepsilon||\gamma^V(t)||^2_{L^2(\mathbb{R}^N)} +C||\gamma^V(t)||^2_{L^2(\mathbb{R}^N))}\nonumber\\
&-& \left[ \varepsilon||\gamma^V(t)||^2_{L^2(\mathbb{R}^N)} +  \varepsilon||\gamma^X(t)||^p_{L^p(\mathbb{R}^N)}  +2C_\varepsilon||\gamma^V(t)||^p_{L^p(\mathbb{R}^N)} \right]\nonumber\\
&\geq& I_\infty(\gamma^X(t)) +C||\gamma^V(t)||_{L^2(\mathbb{R}^N))}^2 -\varepsilon C - 2C_\varepsilon||\gamma^V(t)||^p_{L^2(\mathbb{R}^N)}\nonumber\\
&>& I_\infty(\gamma^X(t)) + \dfrac{C}{2}||\gamma^V(t)||^2_{L^2(\mathbb{R}^N)} - \varepsilon C\nonumber\\
&>& I_\infty(\gamma^X(t)),
\end{eqnarray}
}where we the constants $C,C_\varepsilon$ can change from one line to another, are uniform in $t\in[0,1]$, and we are using that $V=E^-$ is finite dimensional, hence all norms are equivalent on $V$. Moreover, we obtain $\theta_t$ by applying the Mean Value Theorem and subsequently we consider the growth of $f_0$, hence $p>2$. It is important to highlight that we are strongly using the exponential decay of $\tilde{u}$ at infinity.
Therefore, combining (\ref{s4e1})-(\ref{s4e4}) we arrive at
\[ 
c\leq \sup_{u\in Q}I(h_0(u)) \leq \sup_{t\in[0,1]}I(\gamma^X(t)) < \sup_{t\in[0,1]}I_\infty(\gamma^X(t))\leq \sup_{t\in[0,1]}I_\infty(\gamma(t)) = I_\infty(\tilde{u}) \leq c
 \]
 which provides an absurd. Thus, $u \not = 0$ is a nontrivial critical point for $I$, in case $V_0(x)\not\equiv V_{0,\infty}$.
 
\qquad In case $V_0(x)\equiv V_{0,\infty}$, we are looking at the ``problem at infinity" and  provided that \ ${\sigma(A_0)= [V_{0,\infty}, +\infty)}$, \ from $(V_1)_0$ \ we have $\sigma^+ = V_\infty$, hence ${(0, a_0)\cap [V_{0,\infty}, +\infty)\not= \emptyset}$, then ${a_0>V_{0,\infty}>0}$. Thus, if we set $H(s) := F_0(s) - \dfrac{1}{2}V_{0,\infty}s^2$, it satisfies $H(s_0)>0$ for sufficiently large $s_0>0$ and hence we are able to apply Proposition \ref{P4.1} to guarantee the existence of a nontrivial critical point for $I$. Therefore, in all cases, we obtain a nontrivial critical point for $I$, which is a nontrivial solution to problem (\ref{prob_0}).
\end{proof}

\end{document}